\documentclass{amsart}

\usepackage{amssymb}

\newcommand{\CC}{{\mathbb{C}}}
\newcommand{\NN}{{\mathbb{N}}}
\newcommand{\RR}{{\mathbb{R}}}
\newcommand{\TT}{{\mathbb{T}}}
\newcommand{\ZZ}{{\mathbb{Z}}}
\newcommand{\cF}{{\mathcal{F}}}
\newcommand{\cS}{{\mathcal{S}}}
\newcommand{\cQ}{{\mathcal{Q}}}
\newcommand{\Ideal}[1]{{\left\langle #1 \right\rangle}}
\newcommand{\Span}{{\mathop{\rm span}\,}}
\newcommand{\rank}{{\mathop{\rm rank}\,}}
\newcommand{\diag}{{\mathop{\rm diag}\,}}

\newtheorem{theorem}{Theorem}[section]

\newtheorem{proposition}[theorem]{Proposition}
\newtheorem{corollary}[theorem]{Corollary}

\theoremstyle{definition}
\newtheorem{definition}[theorem]{Definition}

\theoremstyle{remark}
\newtheorem{remark}[theorem]{Remark}

\numberwithin{equation}{section}

\begin{document}

\title[Hankel operators of finite rank and Prony's problem]{Hankel and
  Toeplitz operators of finite rank and Prony's problem in several
  variables} 


\author{Tomas Sauer}
\address{Lehrstuhl f\"ur Mathematik mit Schwerpunkt Digitale
  Bildverarbeitung \& FORWISS, University of Passau, Fraunhofer IIS
  Research Group for Knowledge Based Image Processing, Innstr. 43,
  D--94032 Passau, Germany}
\email{Tomas.Sauer@uni-passau.de}



\date{May 21, 2018}

\dedicatory{}

\commby{}

\begin{abstract}
  Prony's problem in several variables has attracted some attention
  recently and provides an interesting combination of polynomial ideal
  theory with analytic and numeric computations. This note points out
  further connections to Hankel operators of finite rank as they
  appear in multidimensional moment problems, shift invariance signal
  spaces, annihilating ideals of filters and factorization of the Hankel
  matrices and operators by means of Vandermonde matrices. In fact, it
  turns out that these concepts are essentially equivalent.
\end{abstract}

\maketitle

\section{Introduction}
\label{sec:Intro}
In 1795, Prony \cite{prony95:_essai} gave an ingenious trick to
recover an exponential sum
$$
f(x) := \sum_{j=1}^n f_j \, \rho_j^x, \qquad f_j, \, \rho_j \in \CC,
$$
from $2n+1$ consecutive integer samples $f(0),\dots,f(2n)$. His method
consists of finding a nontrivial solution $p \in \CC^{n+1}$ of the
homogeneous problem
\begin{equation}
  \label{eq:PronyHomEq}
  \begin{pmatrix}
    f(0) & \dots & f(n) \\
    \vdots & \ddots & \vdots \\
    f(n) & \dots & f(2n)
  \end{pmatrix} \,
  \begin{pmatrix}
    p_0 \\ \vdots \\ p_n
  \end{pmatrix}
  = 0,
\end{equation}
whose associated \emph{Prony polynomial} $p(x) = p_0 + p_1 \, x +
\cdots + p_n \, x^{n}$ has the zeros $\rho_1,\dots,\rho_n$, which
recovers the nonlinear part of $f$; the coefficients $f_j$ can be
found by solving a linear system, cf. \cite{plonka14:_prony}. Emerging
from the classical MUSIC \cite{schmidt86:_multip} and ESPRIT
\cite{roy89:_esprit} algorithms, the numerical behavior of Prony and
Prony-like methods and their relationship to techniques from Numerical
Linear Algebra have been studied carefully,
see, for example, \cite{potts15:_fast_esprit} and the references there.

The multivariate version of Prony's problem has been considered only
recently, with first attempts given in
\cite{kunis16:_prony}, mostly motivated by the connections to
\emph{superresolution}. For the formulation of Prony's problem in $s$
variables, we follow the nowadays popular fashion to write it as an
\emph{exponential reconstruction problem}, i.e., the reconstruction of
a function of the form
\begin{equation}
  \label{eq:PronyExponential}
  f(x) = \sum_{\omega \in \Omega}^n f_\omega \, e^{\omega^T x}, \qquad
  f_\omega \in \CC, \qquad \Omega \subset \left( \RR + i \TT \right)^s,
  \quad \# \Omega < \infty,
\end{equation}
where $\TT = \RR / (2\pi \ZZ)$ denotes the \emph{torus}. We will see
later that for a complete theory the coefficients $f_j$ have to be
chosen from $\Pi = \CC [z]$, the ring of all polynomials in $s$ variables.

The relation to Hankel operators is obvious: the matrix in
(\ref{eq:PronyHomEq}) is a Hankel matrix. Since the situation is
more intricate in the multivariate case, we will define
(generalized) Hankel matrices in a more generous way. To that end, we
let $\ell_0 (\ZZ^s)$ 
denote the space of all sequences $f : \ZZ^s \to \CC$ such that the
``$0$--norm''
$$
\| f \|_0 := \# \left\{ \alpha \in \ZZ^s : f(\alpha) \neq 0 \right\}
$$
is finite. Then, for $A,B \subset \ZZ^s$, $\#A, \# B < \infty$, the
(generalized) \emph{Hankel matrix} is defined as
\begin{equation}
  \label{eq:HankelDef}
  H_{A,B} (f) :=
  \begin{pmatrix}
    f(\alpha+\beta) :
    \begin{array}{c}
      \alpha \in A \\ \beta \in B
    \end{array}
  \end{pmatrix}
  \in \CC^{A \times B}, \qquad f \in \ell_0 (\ZZ^s).
\end{equation}
It is common to only admit $A,B \in \NN_0^s$ for Hankel
matrices and we will see that for our purpose here this makes no
difference. In the same
way, a \emph{Toeplitz matrix} can be defined as
\begin{equation}
  \label{eq:ToelitzDef}
  T_{A,B} (f) :=
  \begin{pmatrix}
    f(\alpha-\beta) :
    \begin{array}{c}
      \alpha \in A \\ \beta \in B
    \end{array}
  \end{pmatrix}
  \in \CC^{A \times B}, \qquad f \in \ell_0 (\ZZ^s).
\end{equation}
Both matrices depend on finitely many values of $f$ on the subsets $A+B$
and $A-B$ of $\ZZ^s$.

\begin{remark}
  (\ref{eq:HankelDef}) and (\ref{eq:ToelitzDef}) are a slightly
  nonstandard way to
  index matrices, but it is the one that captures the structure of
  this matrix. Clearly, $H_{A,B} (f)$ can be written as a conventional
  matrix by
  ordering the multiindices, for example with respect to the graded
  lexicographical ordering, but the resulting matrix is neither a
  Hankel matrix nor does it have any visible structure at all.
\end{remark}

\noindent
The most prominent occurrence of Hankel matrices is probably in the
context of \emph{moment problems},
cf. \cite{schmuedgen17:_momen_probl}, where the Hankel matrix formed
from the moment sequence $\mu (\alpha) = \int x^\alpha \, d\mu$
reveals information about the underlying measure $d\mu$. Also in this
case, multivariate Hankel matrices are naturally multiindexed.

In \cite{Sauer17_Prony} it has been shown that Prony's method
generalizes naturally to several variables if one takes into account
two major points: one has to choose a set $A \in \ZZ^s$ such that
$(\cdot)^A := \Span \{ (\cdot)^\alpha : \alpha \in A \}$ allows for
interpolation at $e^\Omega = \{ e^\omega : \omega \in \Omega \} \subset
\CC^s$ and find some $k$ such that
$$
\rank H_{A,\Gamma_k} (f) = \rank H_{A,\Gamma_{k+1}} (f), \qquad
\Gamma_k := \left\{ \alpha \in \NN_0^s : |\alpha| \le k \right\}.
$$
Then the solutions $h \in \CC^{\Gamma_{k+1}}$ of $H_{A,\Gamma_{k+1}}
(f) \, h = 0$ generate, interpreted as polynomials, the so called
\emph{Prony ideal}, a zero dimensional ideal whose associated variety
equals $e^\Omega$. Moreover, it is shown in \cite{Sauer17_Prony} how
this set can be determined using methods from Numerical Linear
Algebra; \cite{Sauer18:_prony} gives a symbolic method and a canonical
way to determine $A$ if only $\# \Omega$ is known.

In the context of this paper, we consider $H_{A,B} (f)$ as the
\emph{restriction} of the \emph{Hankel operator} $H (f) : \ell_0
(\ZZ^s) \to \ell (\ZZ^s)$ defined as the \emph{correlation}
\begin{equation}
  \label{eq:HankelOp}
  H(f) \, g = \sum_{\beta \in \ZZ^s} f( \cdot + \beta ) \, g(\beta) =:
  f \star g, \qquad g \in \ell_0 (\ZZ^s),
\end{equation}
while the respective \emph{Toeplitz operator} $T(f) \, g = f * g$ represents
the \emph{convolution}. The requirement $g \in \ell_0 (\ZZ^s)$ serves
the purpose of making (\ref{eq:HankelOp}) well--defined. We relate to
$f \in \ell_0 (\ZZ^s)$ its \emph{symbol}
$$
\hat f (z) := \sum_{\alpha \in \ZZ^s} f(\alpha) \, z^\alpha \in
\Lambda := \CC [z,z^{-1}],
$$
which is a \emph{Laurent polynomial} since the support of $f$ is
finite. By means of the translation operator $\tau$, defined by
$\tau_j f = f( \cdot + \epsilon_j )$, $j=1,\dots,s$, where
$\epsilon_j$ are the unit multiindices, and $\tau^\alpha =
\tau_1^{\alpha_1} \cdots \tau_s^{\alpha_s}$, we can rewrite
(\ref{eq:HankelOp}) as
$$
f \star g = \sum_{\beta \in \ZZ^s} \tau^\beta f (\cdot) \, g(\beta)
= \hat g (\tau) \, f =: (f,\hat g),
$$
introducing the bilinear mapping $(\cdot, \cdot) : \ell (\ZZ^s)
\times \Lambda \to \ell (\ZZ^s)$. In the same way we get
$$
f * g = \sum_{\beta \in \ZZ^s} \tau^{-\beta} f (\cdot) \, g(\beta)
= \left( f,\hat g \left( (\cdot)^{-1} \right) \right).
$$
The simple computation
$$
\left( \tau^\alpha f, \hat g \right) = \sum_{\beta \in \ZZ^s}
\tau^{\alpha + \beta} f (\cdot) \, g(\beta)
= \left( (\cdot)^\alpha \, \hat g \right) \, f
$$
leads to the almost trivial but very useful duality
\begin{equation}
  \label{eq:bilinDuality}
  \left( \tau^\alpha f, \hat g \right) = \left( f, (\cdot)^\alpha \hat
    g \right),
\end{equation}
which immediately results in the following observation.

\begin{proposition}\label{P:Invariances}
  Let $f \in \ell (\ZZ^s)$ and $\hat g \in \Lambda$.
  \begin{enumerate}
  \item The linear space $\ker (\cdot,\hat) := \left\{ c \in \ell
      (\ZZ^s) : (c,\hat g) = 0 \right\}$ is \emph{shift invariant},
    i.e., closed under translations.
  \item The linear space $\ker (f,\cdot) := \left\{ q \in \Lambda :
      (f,q) = 0 \right\}$ is a \emph{Laurent ideal}.
  \end{enumerate}
\end{proposition}

\begin{definition}
  A subspace $\cF \subseteq \ell (\ZZ^s)$ is called \emph{shift
    invariant} if $f \in \cF$ implies that $\tau^\alpha f \in \cF$,
  $\alpha \in \NN_0^s$. Moreover,
  for $f \in \ell (\ZZ^s)$ the \emph{shift invariant space}
  $\cS (f)$ generated by $f$ is defined as
  $$
  \cS (f) := \Span \{ \tau^\alpha f : \alpha \in \ZZ^s \}.
  $$
\end{definition}

\noindent
We end this section by defining the rank of Hankel and Toeplitz
operators, setting
\begin{eqnarray}
  \label{eq:rankH}
  \rank H(f) & := & \sup \left\{ \rank H_{A,B} (f) : A,B \subset \ZZ^s,
                    \, \# A, \# B < \infty \right\}, \\
  \label{eq:rankT}
  \rank T(f) & := & \sup \left\{ \rank T_{A,B} (f) : A,B \subset \ZZ^s,
                    \, \# A, \# B < \infty \right\}, \\
  \label{eq:rank+H}
  \rank_+ H(f) & := & \sup \left\{ \rank H_{A,B} (f) : A,B \subset \NN_0^s,
  \, \# A, \# B < \infty \right\}.
\end{eqnarray}
A Hankel or Toeplitz operator is said to be of \emph{finite rank} if
$\rank H(f) < \infty$ or $\rank T(f) < \infty$, respectively.

In the rest of the paper, we will study properties of of multivariate
finite rank Hankel operators and relate them to shift invariant
spaces and zero dimensional ideals. To that end,
Section~\ref{sec:results} will present the main results and the
concepts needed to understand these results. The proofs and further
background material will then be provided in
Section~\ref{sec:Proofs}. Finally, Section~\ref{sec:Conclude} will
provide a short conclusion.

\section{Main results}
\label{sec:results}

We begin by noting that the seemingly different ways of defining the rank
of the operators in (\ref{eq:rankH})--(\ref{eq:rank+H})
lead all to the same number and can even be obtained by the square
symmetric matrices $H_k (f) := H_{\Gamma_k,\Gamma_k} (f)$.

\begin{theorem}\label{T:rankDefsEquiv}
  For $f \in \ell (\ZZ^s)$ we have that
  \begin{equation}
    \label{eq:rankDefsEquiv}
    \rank H(f) = \rank T(f) = \rank_+ H (f) = \lim_{k \to \infty} H_k (f).
  \end{equation}
\end{theorem}

\noindent
In fact, this number is also directly connected to the shift invariant
spaces.

\begin{theorem}\label{T:ranksis}
  For $f \in \ell (\ZZ^s)$ we have that $\rank H(f) = \dim \cS (f)$.
\end{theorem}

\noindent
Since shifts of finitely supported sequences are linearly
independent if the shifts are so large that the supports are disjoint,
Theorem~\ref{T:ranksis} has an immediate consequence.

\begin{corollary}
  If $f \in \ell_0 (\ZZ^s)$ then $\rank H(f) = \infty$.
\end{corollary}

\noindent
It is even possible to give a slightly more ``quantitative'' version
of Theorem~\ref{T:rankDefsEquiv} for Hankel operators of finite
rank. To formulate it, recall the positive part of the
\emph{hyperbolic cross},
\begin{equation}
  \label{eq:HyperDef}
  \Upsilon_n := \left\{ \alpha \in \NN_0^s : \prod_{j=1}^s ( 1 +
    \alpha_j ) \le n \right\},
\end{equation}
which is a canonical and to some extent minimal choice for the set $A$
in Prony's
problem as $(\cdot)^{\Upsilon_n}$ allows for interpolation at
arbitrary $n+1$ points in $\CC^s$, cf. \cite{Sauer18:_prony}. The next
statement tells us when the ranks stabilize.

\begin{theorem}\label{T:RankStable}
  If $\rank H(f) < \infty$, then
  \begin{equation}
    \label{eq:RankStable}
    \rank H_k (f) = \rank H_{\Upsilon_k,\Upsilon_k} = \rank H(f), \qquad
    k \ge \rank H(f).
  \end{equation}
\end{theorem}

\noindent
The next theorem connects finite rank Hankel and Toeplitz operators to
ideals and Prony's problem. To that end, recall that an \emph{ideal} $I$
in $\Lambda$ or $\Pi$ is a subset that is closed under addition and
multiplication with arbitrary elements of $\Lambda$ and $\Pi$,
respectively. Laurent ideals are somewhat intricate since they are
only well-defined on $(\CC \setminus \{ 0 \})^s$ and since there
exists only a trivial grading on $\Lambda$; but already the
proofs in \cite{Sauer17_Prony} showed that we can easily restrict ourselves
to \emph{polynomial ideals} $I \subseteq \Pi$.

An ideal is called \emph{zero dimensional} if $\Pi / I$
is finite dimensional which also implies that the associated variety,
$V(I) = \{ z \in \CC^s : f(z) = 0, f \in I \}$, is finite. Any
polynomial ideal has a finite basis $G$, i.e., a finite subset $G
\subset I$ such that
$$
I = \Ideal{G} = \left\{ \sum_{g \in G} q_g \, g : q_g \in \Pi, \, g
  \in G \right\}.
$$
A special choice for such a basis are the well-known \emph{Gr\"obner
  bases} which can be computed efficiently and allow for a
well-defined computation of division with unique
remainder. cf. \cite{CoxLittleOShea92}. The ideal theoretic approach
to solve Prony's problem then leads to the following result.

\begin{theorem}\label{T:finrankIdealChar}
  For $f \in \ell (\ZZ^s)$ the following statements
  are equivalent. 
  \begin{enumerate}
  \item\label{it:finrankIdealChar1} $\rank H(f) < \infty$.
  \item\label{it:finrankIdealChar2} There exists an ideal $I \subset
    \Pi$ with a Gr\"obner basis $G$ such that
    \begin{equation*}
      \cS (f) = \bigcap_{q \in I} \ker \, ( \cdot,q ) = \bigcap_{\hat g
        \in G} \ker \, ( \cdot,\hat g )
    \end{equation*}
    and $\rank H(f) = \dim \Pi / I$.
  \item\label{it:finrankIdealChar3}
    There exists $\Omega \subset \left( \RR + i \TT \right)^s$,
    $\# \Omega < \infty$ and shift invariant subspaces $\cQ_\omega
    \subset \Pi$, $\omega \in \Omega$, such that
    \begin{equation*}
      \cS (f) = \bigoplus_{\omega \in \Omega} \cQ_\omega \,
      e^{\omega^T \cdot} \qquad \text{and} \qquad \rank H(f) =
      \sum_{\omega \in \Omega} \dim \cQ_\omega. 
    \end{equation*}
  \item\label{it:finrankIdealChar4}
    $f$ is of the form
    \begin{equation*}
      f (x) = \sum_{\omega \in \Omega} f_\omega (x) \, e^{\omega^T x},
      \qquad f_\omega \in \cQ_\omega, \quad \omega \in \Omega,
    \end{equation*}
    where $\Omega$ and $\cQ_\omega$ are as in
    (\ref{it:finrankIdealChar3}).
  \end{enumerate}
\end{theorem}

\begin{remark}
  The ideal $I$ of statement (\ref{it:finrankIdealChar2}) is
  the \emph{annihilating filter ideal} of the shift invariant space
  $\cS(f)$. Strictly speaking, filters are usually defined as
  convolutions, but since a convolution is just a correlation with the
  reflection of the filter, this makes no difference. Alternatively,
  one could also consider the Gr\"obner basis $G$ as a system of
  partial difference equations whose homogeneous solution space is
  again $\cS(f)$.
\end{remark}

\noindent
The simplest case of the representation given in statement
(\ref{it:finrankIdealChar4}) of Theorem~\ref{T:finrankIdealChar} is
that all spaces $\cQ_\omega$ are simplest possible, i.e., $\cQ_\omega
= \Pi_0 = \CC$. This corresponds to the generic situation that all
the common zero are simple, see
Theorem~\ref{T:Groebner} and the discussion following it, and deserves
to be distinguished.

\begin{definition}
  A Hankel operator $H(f)$ of finite rank is called \emph{simple} if
  $\dim \cQ_\omega = 1$, $\omega \in \Omega$.
\end{definition}

\noindent
Simple Hankel operators, i.e., Hankel operators formed from
multiinteger samples of functions of the form
\begin{equation}
  \label{eq:HankelSimple}
  f(x) = \sum_{\omega \in \Omega} f_\omega \, e^{\omega^T x}, \qquad
  f_\omega \in \CC \setminus \{ 0 \},
\end{equation}
admit a particularly simple factorization that is obtained very easily. To
that end, recall the concept of the \emph{Vandermonde matrix} to a
\emph{Lagrange interpolation problem} at $\Theta$, 
$$
V(\Theta;A) =
\begin{pmatrix}
  \theta^\alpha :
  \begin{array}{c}
    \theta \in \Theta \\ \alpha \in A
  \end{array}
\end{pmatrix} \in \CC^{\Theta \times A}, \qquad \Theta \subset \CC^s, \quad
A \subseteq \NN_0^s,
$$
which allows us allows to write the interpolation problem
$$
y_\theta = \sum_{\alpha \in A} a_\alpha \, \theta^\alpha, \qquad
\theta \in \Theta,
$$
as the linear system $y = V( \Theta;A ) \, a$. If $f$ is of the form
(\ref{eq:HankelSimple}), we get for $\alpha,\beta \in \NN_0^s$ that
$$
H(f)_{\alpha,\beta}
=  \left( H_{A,B} (f) \right)_{\alpha,\beta} 
= \sum_{\omega \in \Omega} f_\omega \, e^{\omega^T (\alpha+\beta)}
= \left( V (e^\Omega;A)^T F_\Omega V(e^\Omega;B) \right)_{\alpha,\beta},
$$
where $F_\Omega = \diag
\begin{pmatrix}
  f_\omega : \omega \in \Omega
\end{pmatrix}$
This already proves the following result.

\begin{corollary}
  $H(f)$ is a simple Hankel operator of finite rank if and only if
  there exists a nonsingular diagonal matrix $F_\Omega \in \CC^{\rank
    H(f) \times \rank H(f)}$ such that
  \begin{equation}
    \label{eq:FinrankHFact}
    H_{A,B} (f) = V (e^\Omega;A)^T F_\Omega V(e^\Omega;B), \qquad
    H(f) = V \left( e^\Omega;\NN_0^s \right)^T F_\Omega V \left(
      e^\Omega;\NN_0^s \right).
  \end{equation}
\end{corollary}

\noindent
The factorizations (\ref{eq:FinrankHFact}) are known in various
instances and play a fundamental role in the multidimensional
(truncated) moment problem, cf. \cite{schmuedgen17:_momen_probl}.

\noindent
In the general situation, the analogy of (\ref{eq:FinrankHFact}) is
slightly more
intricate since now multiple zeros have to be considered. In the
context of Prony's problem this has been first done in
\cite{Mourrain16P}; a different approach has been studied in
\cite{Sauer2017:_Reconstruction}. To recall the latter, let $\Theta
\subset \CC^s$
be a finite set of nodes, let $\cQ_\Theta = \left( \cQ_\theta : \theta
  \in \Theta \right)$ be a vector of \emph{$D$--invariant} multiplicity
spaces which will be defined precisely in
Definition~\ref{D:DInvariant}, and let $Q_\theta \subset \Pi$, $\#
Q_\theta = \dim \cQ_\theta$, $\theta \in \Theta$, be bases of these
multiplicity spaces. Then the Vandermonde matrix
\begin{equation}
  \label{eq:HermiteVandermonde}
  V \left( \Theta,Q_\Theta; A \right) :=
  \begin{pmatrix}
    \left( q(D) (\cdot)^\alpha \right) (\theta) :
    \begin{array}{c}
      q \in Q_\Theta, \, \theta \in \Theta \\ \alpha \in A
    \end{array}
  \end{pmatrix}  
\end{equation}
encodes the Hermite interpolation problem (\ref{eq:HermiteIntProb})
which will be discussed later as well. Moreover, it allows us to give
the general factorization of finite rank Hankel operators.

\begin{corollary}\label{C:FactorThm}
  $H(f)$ is a finite rank Hankel operator if and only if there exists
  a finite set $\Omega \subset \left( \RR + i \TT \right)^s$, finite
  dimensional $D$-invariant spaces $\cQ_\Theta \subset \Pi$ with basis
  $Q_\Theta$, $\theta \in \Theta$, and a nonsingular block diagonal
  matrix 
  \begin{equation}
    \label{eq:FactorThmFMat}
    F_{e^\Omega,Q_\Omega} =
    \begin{pmatrix}
      F_\omega \in \CC^{\dim \cQ_\omega \times \dim \cQ_\omega} : \omega \in \Omega 
    \end{pmatrix}
  \end{equation}
  such that
  \begin{equation}
    \label{eq:FactorThm}
    H_{A,B} (f) = V \left( \Theta,Q_\Theta; A \right)^T
    F_{e^\Omega,Q_\Omega} \, V \left( \Theta,Q_\Theta; B \right),
    \qquad A,B \subseteq \NN_0^s,
  \end{equation}
  and, in particular,
    \begin{equation}
    \label{eq:FactorThm2}
    H (f) = V \left( \Theta,Q_\Theta; \NN_0^s \right)^T
    F_{e^\Omega,Q_\Omega} \, V \left( \Theta,Q_\Theta; \NN_0^s \right).
  \end{equation}
\end{corollary}

\section{Proofs, background and auxiliary results}
\label{sec:Proofs}
We first note that since $T_{A,B} (f) = H_{A,-B} (f)$, the first two
numbers in (\ref{eq:rankDefsEquiv}) coincide trivially. However, we
begin with the proof of Theorem~\ref{T:ranksis}.

\begin{proof}[Proof of Theorem~\ref{T:ranksis}]
  For $n \le \dim \cS (f)$ and let $f_j = \tau^{\alpha^j} f$,
  $\alpha^j \in \NN_0^s$, $j=1,\dots,n$, be linearly independent
  elements of $\cS (f)$. For $c \in \CC^n$ define $g  \in \ell
  (\ZZ^s)$ as $g := \overline{c_1 \, f_1} + \cdots + \overline{c_n \,
    f_n}$, choose $B$ such that $g(B) \neq 0$ and $A$ such that
  $\bigcup_{j=1}^n \alpha^j \subseteq A$. Then
  \begin{eqnarray*}
    \sum_{j=1}^n c_j \left( H_{A,B} (f) \, g
    \right)_{\alpha^j}
    & = & \sum_{j=1}^n c_j \sum_{\beta \in B} f (\alpha^j + \beta)
          \sum_{k=1}^n \overline{c_k \, f_k (\beta)} \\
    & = & \sum_{j,k=1}^n c_j \overline{c_k} \sum_{\beta \in B} f
          (\alpha^j + \beta) \overline{f (\alpha^k + \beta)}
          = \| g(B) \|_2^2 > 0.
  \end{eqnarray*}
  In other words, $H_{A,B} (f) \, g \neq 0$ for any $g \in \Span \{
  \overline{f_j (B)}  : j=1,\dots,n \}$. This shows that
  $$
  n \le \dim \cS (f) \qquad \Rightarrow \qquad n \le \rank H_{A,B}
  (f) \le \rank (f),
  $$
  implying $\dim \cS (f) \le \rank H (f)$, also in the case $\dim \cS
  (f) = \infty$.

  Conversely, suppose first that $\rank H(f) < \infty$ and choose
  $A,B$ so large that $\rank H_{A,B} (f) = \rank H(f) =:n$. Then
  $H_{A,B} (f)$ contains $n$ linearly independent rows with indices
  $\alpha^j$, $j=1,\dots,n$, so that the sequences $\tau^{\alpha^j}
  f$, $j=1,\dots,n$, are linearly independent even on $B$, and $\dim
  \cS (f) \ge \rank H(f)$. If $\rank H(f) = \infty$, the same argument
  applied to sequences $A_j, B_j$ such that
  $$
  \infty = \rank H(f) = \lim_{j \to \infty} H_{A_j,B_j} (f)
  $$
  shows that $\dim \cS(f) = \infty$.
\end{proof}

\noindent
The shift invariant spaces allow us to complete the proof of
Theorem~\ref{T:rankDefsEquiv}.

\begin{proof}[Proof of Theorem~\ref{T:rankDefsEquiv}]
  Since $\rank_+ H(f) \le \rank H(f)$ and $\rank H_k (f) \le \rank
  H(f)$, it suffices to prove that $\rank_+ H(f) \ge \dim \cS (f)$ and
  that
  $$
  \lim_{k \to \infty} H_k (f) \ge \dim \cS (f).
  $$
  This will be done by choosing the $\alpha^j$ in the preceding proof
  appropriately by taking into account that the $\tau^{\alpha^j} f$
  are linearly independent if and only if $\tau^{\alpha^j + \beta}
  f$, $j=1,\dots,n$, are linearly independent for any $\beta \in
  \ZZ^s$. Moreover, $\cS(f) = \cS( \tau^\gamma f )$ for any $\gamma
  \in \ZZ^s$. We now only have to choose $\gamma \in \ZZ^s$ such that
  $g(B) \neq 0$ for some $B \subset \NN_0^s$ and then $\beta \in
  \NN_0^s$ such that $\alpha^j + \beta \in \NN_0^s$. The same argument
  also works for $H_k$.
\end{proof}

\noindent
The first step of the proof of Theorem~\ref{T:finrankIdealChar} uses
the algebraic solution method for Prony's method: the kernel of a
sufficiently large Hankel submatrix of $H(f)$ defines a polynomial ideal,
the so--called \emph{Prony ideal} whose common zeros are $e^\Omega$
and thus yield the frequencies.

\begin{proof}[Proof of Thorem~\ref{T:finrankIdealChar},
  (\ref{it:finrankIdealChar1}) $\Rightarrow$ (\ref{it:finrankIdealChar2})]
  Since $H(f)$ is of finite rank, there exists, by
  Theorem~\ref{T:rankDefsEquiv} a minimal $n \in \NN$ such that $\rank H(f)
  = \rank H_k (f)$ for any $k \ge n$. If $H_k p = 0$ for some $p \in
  \Gamma_k$, then $\hat p$ is a polynomial such that $(f,\hat p) =
  0$. By Proposition~\ref{P:Invariances}, the set of all $\hat p$ such
  that $H_k p = 0$ for some $k \ge n$ forms an ideal $I$ which has a
  finite Gr\"obner basis $G$ such that $I = \Ideal{G}$. Hence,
  $$
  f \in \bigcap_{\hat g \in G} \ker \, (\cdot,\hat g)
  $$
  and since, again by Proposition~\ref{P:Invariances}, these kernels
  are shift invariant, it follows that $\cS (f) \subseteq \bigcap_{\hat g
    \in G} \ker \, (\cdot,\hat g)$. Since, in addition
  \begin{eqnarray*}
    \dim \ker \, (\cdot,\hat g)
    & = & \dim \Pi / I = {k+s \choose k} - \dim \ker H_k (f) = \rank
          H_k (f) \\
    & = & \rank H(f) = \dim \cS(f) 
  \end{eqnarray*}
  by Theorem~\ref{T:rankDefsEquiv}, we can finally conclude that $\cS
  (f) = \bigcap_{\hat g \in G} \ker \, (\cdot,\hat g)$.
\end{proof}

\begin{corollary}\label{C:KernIdealZero}
  For the ideal $I$ we have that $z \in V(I)$ implies $z_j \neq 0$,
  $j=1,\dots,s$.
\end{corollary}

\begin{proof}
  Suppose that there exists $z \in V(I)$ with $z_j = 0$ for some $j
  \in \{1,\dots,s\}$. Then $z \not\in V \left( I : \Ideal{ (\cdot)_j
    } \right)$. Choose any $q \in I : \Ideal{ (\cdot)_j}$, i.e.,
  $(\cdot)_j q \in I$, then $(\cS(f),I) = 0$ yields that
  $$
  0 = \left( \cS(f),(\cdot)_j q \right) = \left( \tau^{\epsilon_j}
    \cS(f), q \right) = \left( \cS(f), q \right).
  $$
  By the preceding proof of Thorem~\ref{T:finrankIdealChar},
  (\ref{it:finrankIdealChar1}) $\Rightarrow$
  (\ref{it:finrankIdealChar2}), this then yields the contradiction
  $$
  \dim \cS (f) = \dim \Pi / I > \dim \Pi / \left( I : \Ideal{
      (\cdot)_j} \right) = \dim \cS (f).
  $$
  Hence, $V(I) \subset \left( \CC \setminus \{ 0 \} \right)^s$.   
\end{proof}

\begin{remark}
  The requirements
  $$
  0 = (f,\hat g) = f \star g, \qquad \hat g \in G
  $$
  yield a \emph{system of homogeneous difference equations} to
  determine $f$, or more precisely a shift invariant space of
  solutions. In this respect determining the Prony ideal can be
  formulated in the language of signal processing as determining a
  system of \emph{annihilating filters} for the signal $f$.
\end{remark}

\noindent
The next step in the proof of Theorem~\ref{T:finrankIdealChar}
requires some more
background. To that end, recall that any zero dimensional ideal
has finitely many zeros, say $\Theta \subset \CC^s$, $\# \Theta <
\infty$; as shown by Gr\"obner
\cite{groebner37:_ueber_macaul_system_bedeut_theor_differ_koeff,groebner39:_ueber_eigen_integ_differ_koeff}, 
the \emph{multiplicities} of these zeros are not mere numbers any more,
but structural quantities.

\begin{definition}\label{D:DInvariant}
  A subspace $\cQ \subseteq \Pi$ of polynomials is called
  \emph{$D$--invariant}, if $q \in \cQ$ also implies that $p(D) q \in
  \cQ$ for any $p \in \Pi$, where, as usually,
  $$
  p(D) = \sum_{\alpha \in \NN_0^s} p_\alpha \,
  \frac{\partial^{|\alpha|}}{\partial x^\alpha}, \qquad p(z) =
  \sum_{\alpha \in \NN_0^s} p_\alpha \, z^\alpha,
  $$
  denotes the differential operator induced by $p$.
\end{definition}

\noindent
This notion allows us to formulate Gr\"obner's result on multiple
common zeros; more recent work on multiplicities in polynomial system
solving can be found in \cite{MarinariMoellerMora96}.

\begin{theorem}[Gr\"obner,
  \cite{groebner37:_ueber_macaul_system_bedeut_theor_differ_koeff}]\label{T:Groebner} 
  $I \subset \Pi$ is a zero dimensional ideal if and only if there
  exist a finite set $\Theta \subset \CC^s$ and $D$--invariant
  subspaces $\cQ_\theta' \subset \Pi$ such that
  \begin{equation}
    \label{eq:GroebnerZeroDim}
    I = \bigcap_{\theta \in \Theta} \left\{ f \in \Pi : q(D) f
      (\theta) = 0, \, q \in \cQ'_\theta \right\}.
  \end{equation}
  Moreover,
  \begin{equation}
    \label{eq:GroebnerZeroDim2}
    \dim \Pi / I = \sum_{\theta \in \Theta} \dim \cQ_\theta'.
  \end{equation}
\end{theorem}

\noindent
As a consequence, it can easily be shown that the \emph{Hermite
  interpolation problem}
\begin{equation}
  \label{eq:HermiteIntProb}
  q(D) f (\theta) = 0, \qquad q \in Q_\theta', \quad \theta \in \Theta,
\end{equation}
where $Q_\theta'$ is a basis of $\cQ_\Theta'$, has a unique solution
in $\Pi / I$, hence the functionals in (\ref{eq:HermiteIntProb}) are
the natural dual functionals for $\Pi / I$. Also note that this
Hermite interpolation problem is an \emph{ideal interpolation} in the
sense of \cite{Birkhoff79,boor05:_ideal}.

Finally, we recall the operator
\begin{equation}
  \label{eq:LDef}
  L : \Pi \to \Pi, \qquad f \mapsto \sum_{\alpha \in \NN_0^s}
  \frac1{\alpha!} ( \tau - I )^\alpha f(0) \, (\cdot)^\alpha,
\end{equation}
from \cite{Sauer16:_kernels,Sauer2017:_Reconstruction}. With the
\emph{Pochhammer symbols} or \emph{falling factorials}
\cite{GrahamKnuthPatashnik98}, defined as 
$$
(\cdot)_\alpha := \prod_{j=1}^s \prod_{\beta_j = 0}^{\alpha_j} \left( (\cdot)_j -
\beta_j \right), \qquad \alpha \in \NN_0^s,
$$
its inverse can be written explicitly as
\begin{equation}
  \label{eq:L-1Def}
  L^{-1} f = \sum_{\alpha \in \NN_0^s}
  \frac1{\alpha!} \frac{\partial^{|\alpha|} f}{\partial x^\alpha} (0)
  \, (\cdot)_\alpha
\end{equation}
The operator allows to switch between shift invariant and
$D$--invariant polynomial subspaces.

\begin{proposition}[\cite{Sauer2017:_Reconstruction},
  Proposition~1]\label{P:DLShiftInvar}
  A subspace $\cQ$ of $\Pi$ is shift invariant if and only if $L\cQ$
  is $D$--invariant.
\end{proposition}

\begin{proof}[Proof of Thorem~\ref{T:finrankIdealChar},
  (\ref{it:finrankIdealChar2}) $\Rightarrow$
  (\ref{it:finrankIdealChar3}) $\Rightarrow$
  (\ref{it:finrankIdealChar4})]
  The zero dimensional ideal $I$ with Gr\"obner basis $G$ defines
  a system of homogeneous \emph{difference equations} via
  \begin{equation}
    \label{eq:HomogenDiffEq}
    0 = f \star g = \sum_{\beta \in \ZZ^s} f(\cdot + \beta) \, g(\beta)
    = \sum_{\beta \in \ZZ^s} f(\cdot - \beta) \, g(-\beta)
    = f * g( - \cdot ), \qquad \hat g \in G,    
  \end{equation}
  which is solved by $\cS (f)$
  Let $\Theta$ and $\cQ_\theta' \subset \Pi$, $\theta \in \Theta$,
  denote the common zeros of $I$ and their multiplicities where
  Corollary~\ref{C:KernIdealZero} ensures that $\Theta \subset \left( \CC
    \setminus \{0\} \right)^s$.
  It has been shown in \cite{Sauer16:_kernels} that all solutions of
  the homogeneous difference equation (\ref{eq:HomogenDiffEq}), or,
  equivalently, all common kernels of the convolution operators
  defined by $g(-\cdot)$, $g \in G$, are of the form
  $$
  f = \sum_{\theta \in \Theta} q_\theta \, \theta^{(\cdot)}, \qquad
  q_\theta \in Q_\theta :=  L^{-1} \sigma_\theta \cQ_\theta', \qquad
  \theta \in \Theta,
  $$
  where
  $$
  \Theta = V \left( \Ideal{\widehat{g(-\cdot)} : \hat g \in G}
  \right)^{-1} = V \left( \Ideal{G} \right),
  $$
  and $\sigma_\theta : f = f \left( \theta_1 (\cdot)_1,\dots,\theta_s
    (\cdot)_s \right)$ denotes the dilation by the diagonal matrix
  formed by $\theta$.
  Writing $\Theta = e^\Omega$ and taking into account that
  $\cQ_\theta$ is shift invariant due to
  Proposition~\ref{P:DLShiftInvar}, gives the desired  
  representation. (\ref{it:finrankIdealChar3}) $\Rightarrow$
  (\ref{it:finrankIdealChar4}) is a direct consequence.
\end{proof}

\noindent
To complete the proof of Theorem~\ref{T:finrankIdealChar}, we recall
from \cite{Sauer2017:_Reconstruction} the factorization theorem for
Hankel operators associated to functions of ``Prony form''.

\begin{theorem}[\cite{Sauer2017:_Reconstruction},
  Theorem~5]\label{T:PronyFactGeneral} 
  If $f$ is of the form
  \begin{equation}
    \label{eq:PronyFormGeneral}
    f(x) = \sum_{\omega \in \Omega} f_\omega (x) \, e^{\omega^T x},
    \qquad f_\omega \in \Pi \setminus \{ 0 \}, \quad \omega \in \left(
      \RR + i \TT \right)^s,  
  \end{equation}
  then there exist $D$--invariant spaces $\cQ_\omega \subset \Pi$ and
  a nonsingular block diagonal matrix
  \begin{equation}
    \label{eq:PronyFactGeneral1}
    F_{\Omega,\cQ_\Omega} :=
    \diag \begin{pmatrix}
      F_\omega \in \CC^{\dim \cQ_\omega \times \dim \cQ_\omega} :
      \omega \in \Omega
    \end{pmatrix}
  \end{equation}
  such that
  \begin{equation}
    \label{eq:PronyFactGeneral2}
    H_{A,B} (f) = V \left( e^\Omega,\cQ_\Omega; A \right)^T
    F_{\Omega,\cQ_\Omega} \, V \left( e^\Omega,\cQ_\Omega; B \right),
    \qquad A,B \subset \NN_0^s.
  \end{equation}
\end{theorem}

\noindent
This allows us to
eventually complete the proof of Theorem~\ref{T:finrankIdealChar}.

\begin{proof}[Proof of Thorem~\ref{T:finrankIdealChar},
  (\ref{it:finrankIdealChar4}) $\Rightarrow$
  (\ref{it:finrankIdealChar1})]
  Statement (\ref{it:finrankIdealChar4}) means that $f$ is of the form
  (\ref{eq:PronyFormGeneral}) and therefore $H_{A,B} (f)$ factorizes
  as in (\ref{eq:PronyFactGeneral2}) for any choice of $A,B \subset
  \NN_0^s$. This implies that
  $$
  \rank H_{A,B} (f) \le \max \left\{ \rank V \left(
      e^\Omega,\cQ_\Omega; A \right), \rank V \left(
      e^\Omega,\cQ_\Omega; B \right), \rank
    F_{\Omega,\cQ_\Omega} \right\},
  $$
  and since
  $$
  \rank V \left( e^\Omega,\cQ_\Omega; A \right) \le \sum_{\omega \in
    \Omega} \dim \cQ_\omega = \rank F_{\Omega,\cQ_\Omega}, \qquad A
  \subseteq \NN_0^s,
  $$
  with equality if and only if $(\cdot)^A$ is an interpolation space
  for the Hermite interpolation problem~(\ref{eq:HermiteIntProb}), it
  follows that $H_{A,B} (f) \le \rank F_{\Omega,\cQ_\Omega}$, again
  with equality iff $(\cdot)^A$ and $(\cdot)^B$ are interpolation
  spaces. Consequently,
  $$
  \rank H(f) = \rank F_{\Omega,\cQ_\Omega} = \sum_{\omega \in
    \Omega} \dim \cQ_\omega < \infty.
  $$
\end{proof}

\noindent
This also proves the factorization theorem,
Corollary~\ref{C:FactorThm}: the necessity of the factorization
follows for a finite rank follows from
Theorem~\ref{T:finrankIdealChar}~(\ref{it:finrankIdealChar4}) and
Theorem~\ref{T:PronyFactGeneral}, its sufficiency was exactly the
point in the proof above.

And we can prove our last remaining result of Section~\ref{sec:results}.

\begin{proof}[Proof of Theorem~\ref{T:RankStable}]
  Since $H(f)$ defines a Hermite interpolation problem with $\rank
  H(f)$ conditions, it follows that $(\cdot)^{\Upsilon_k}$ and
  $(\cdot)^{\Gamma_k}$ admit Hermite 
  interpolation, cf. \cite{Sauer2017:_Reconstruction}. This yields
  that
  $$
  \rank V \left( e^\Omega,\cQ_\Omega; \Upsilon_k \right) = \rank V \left(
    e^\Omega,\cQ_\Omega; \Gamma_k \right) = \rank H(f),
  $$
  hence (\ref{eq:RankStable}).
\end{proof}

\section{Conclusion}
\label{sec:Conclude}
We have seen that Hankel operators of finite rank defined on sequences
are practically equivalent to shift invariant subspaces and to zero
dimensional annihilating ideals where even the rank of the operator,
the dimension of the shift invariant space and the codimension of the
ideal coincide. The connection between these notions is Prony's
problem in its generalized form with polynomial coefficients.


\bibliographystyle{amsplain}
\bibliography{../bibls/4all}

\end{document}